\documentclass[11pt,a4paper]{article}
\usepackage{amssymb,amsmath}
\usepackage{tikz}
\usetikzlibrary{positioning}
\usepackage{amsfonts}

\usetikzlibrary{arrows,patterns}

\usepackage{enumerate}
\usepackage{amssymb}
\usepackage{ textcomp }

\newtheorem{theorem}{Theorem}[section]
\newtheorem{corollary}[theorem]{Corollary}
\newtheorem{definition}[theorem]{Definition}
\newtheorem{lemma}[theorem]{Lemma}
\newtheorem{proposition}[theorem]{Proposition}
\newtheorem{remark}[theorem]{Remark}

\newtheorem{example}[theorem]{Example}

\newenvironment{proof}{\begin{trivlist}\item[]{\it Proof.}}
{\hfill$\square$\end{trivlist}}

\author{M. Zubor}

\title{Semilattice Indecomposable Finite Semigroups With Large Subsemilattices\footnote{This work was supported by the National Research, Development and Innovation Office – NKFIH, 115288. MSC: 20M10. Keywords: semigroups, semilattices, semigroup algebra}}
\date{ }
\begin{document}

\maketitle

\begin{center}
{\small Department of Algebra, Budapest University of Technology and Economics }

{\small 1521 Budapest, P. O. Box 91, Hungary}

{\small E-mail: zuborm@math.bme.hu  } 
\end{center}

\bigskip

\begin{abstract}
In this paper we show that if $Y$ is a subsemilattice of a finite semilattice indecomposable semigroup $S$ then $|Y|\leq 2\left\lfloor \frac{|S|-1}{4}\right\rfloor+1$. We also characterize finite semilattice indecomposable semigroups $S$ which contains a subsemilattice $Y$ with $|S|=4k+1$ and $|Y|=2\left\lfloor \frac{|S|-1}{4}\right\rfloor+1=2k+1$. They are special inverse semigroups. Our investigation is based on our new result proved in this paper which characterize finite semilattice indecomposable semigroups with a zero by only use the properties of its semigroup algebra.
\end{abstract}

\section{Introduction and motivation}
It is known (\cite{Tamura2}) that every semigroup is a semilattice of semilattice indecomposable (s-in\-de\-com\-po\-sab\-le) semigroups. In the literature of the theory of semigroups there are many papers about s-indecomposable semigroups (see, for example, the papers, \cite{Chrislock}, \cite{Nordahl} - \cite{Nagy5},
\cite{PutchaWeissglass}, \cite{Tamura1}, \cite{TamuraKimura}, and the books \cite{Grillet}, \cite{Nagy6}).
Some of them deal with the s-indecomposable semigroups without idempotents, the others investigate the s-indecomposable semigroups containing at least one idempotent. In this paper we deal with finite s-indecomposable semigroups in terms of what can be said about the size of their subsemilattices. The answer is known in special classes of semigroups. 
In the classes of semigroups investigated in \cite{Chrislock}, \cite{Nordahl} - \cite{Nagy5}, \cite{TamuraKimura}, the finite s-indecomposable semigroups are ideal extensions of special completely simple semigroups by nilpotent semigroups. As $ef=fe$ implies $e=f$ for every idempotent elements $e$ and $f$ of a completely simple semigroup, the cardinality of the subsemilattices in a finite s-indecomposable completely simple semigroup is one.

The situation is more interesting in general. In our present paper we show that if $Y$ is a subsemilattice of a finite s-indecomposable semigroup $S$ then $|Y|\leq 2\left\lfloor \frac{|S|-1}{4}\right\rfloor+1$. We also show that there are finite s-indecomposable semigroups $S$ which contain a subsemilattice $Y$ such that $|Y|=2\left\lfloor \frac{|S|-1}{4}\right\rfloor+1$. Moreover, these semigroups are characterized here, when $|S|=4k+1$.

\section{Preliminaries}

Let $S$ be a semigroup. Let ${\mathbb C}[S]$ denote the semigroup algebra of $S$ over the field $\mathbb C$ of all complex numbers. The contracted semigroup algebra of a semigroup $S$ with a zero (over $\mathbb{C}$) will be denoted by ${\mathbb C}_0[S]$ (see \cite[p.35]{Okninski}).

For a finite dimensional algebra $A$ over $\mathbb{C}$, the Jacobson radical of $A$ will be denoted by $J(A)$. 
It is known that $J(A)$ is the set of all properly nilpotent elements of $A$. We will use the following well-known facts: the factor algebra $A/J(A)$ is semisimple (and so, for a finite semigroup $S$, $\mathbb{C}[S]/J(\mathbb{C}[S])$ is semisimple), moreover a finite dimensional algebra $A$ over $\mathbb{C}$ is semisimple if and only if $A$ is isomorphic to $\bigoplus_{i=1}^k M_{n_i}(\mathbb{C})$, where $M_n(\mathbb{C})$ denotes the associative algebra of all $n\times n$ matrices over $\mathbb{C}$.

If a semigroup $S$ has a minimal ideal $K_S$, then $K_S$ is called the kernel of $S$. Every finite semigroup evidently has a kernel. If a semigroup $S$ has a kernel, then $K_S$ is a simple subsemigroup of $S$ \cite[Cor. 2.30. p.69]{CliffordPreston}. Every finite simple [$0$-simple] semigroup is completely simple [completely $0$-simple] by \cite[Cor. 2.56. p.83]{CliffordPreston}.

For Rees matrix semigroups we will use the notation of \cite[p.88]{CliffordPreston}. Rees Theorem \cite[Thm. 3.5. p.94]{CliffordPreston} characterizes the completely simple [completely $0$-simple] semigroups. A semigroup is completely $0$-simple if and only if it is isomorphic with a regular Rees matrix semigroup over a group with a zero.
A completely $0$-simple semigroup is an inverse semigroup if and only if it is a Brandt semigroup \cite[Thm. 3.9. p.102]{CliffordPreston}. 
In our investigation a special type of Brandt semigroups is in the focus. This is the semigroup $\mathcal{M}^0(1;2,2;I)$ where $1$ denote the one-element group and $I$ is the $2\times 2$ identity matrix. We will denote this Brandt semigroup by $B_2$.

\section{Semilattice indecomposable semigroups} 

A semigroup $S$ is said to be 
{\it semilattice indecomposable} (s-indecomposable) if every semilattice homomorphic image of $S$ is trivial (that is, it contains only one element). An ideal $I$ of a semigroup $S$ is called a {\it completely prime ideal} if $S\setminus I$ is a subsemigroup of $S$. It is known (\cite[I.8.3. Prop. p.15]{Petrich}) that a semigroup is s-indecomposable if and only if it does not contain completely prime ideals. Corollary in \cite{Tamura2} gives an other characterization of s-indecomposable semigroups. A semigroup $S$ is s-indecomposable if and only if, for every $a, b\in S$, there is a sequence $a=a_0, a_1, \dots , a_{n-1}, a_n=b$ of elements of $S$ such that $a_{i-1}$ divides some power of $a_i$ ($i=1, \dots , n$).

In Theorem \ref{tetel} we give a new characterization of finite s-indecomposable semigroups $S$ by the terms of semigroup algebras ${\mathbb C}[S/K_S]$. In our investigation we shall use the next lemma, which is a special case of Theorem 4.1 of \cite{Coleman}.

\begin{lemma}\label{direkt}
If $Y$ is a finite semilattice then the algebra ${\mathbb C}[Y]$ is semisimple and isomorphic to the direct sum $\bigoplus _{i\in Y}{\mathbb C}$.
\end{lemma}

\begin{theorem}\label{tetel}
A finite semigroup $S$ is s-indecomposable if and only if, $\mathbb{C}[S/K_S]/J(\mathbb{C}[S/K_S])$ has exactly one $1$-dimensional ideal.
\end{theorem}

\begin{proof}
Let $S$ be a finite semigroup. First we consider
the case when $S$ has a zero $z$. In this case $K_S=\{z\}$ and so $S/K_S\cong S$ .

 Let  $\alpha$ be a semilattice congruence on $S$. There is a $\varphi:\mathbb{C}[S]\rightarrow\mathbb{C}[S/\alpha]$ surjective homomorphism. The algebra $\mathbb{C}[S/\alpha]$ is semisimple by Lemma \ref{direkt}, and so $J(\mathbb{C}[S])\subseteq \ker(\varphi)$. Thus there is a surjective homomorphism $\phi:\mathbb{C}[S]/J(\mathbb{C}[S])\rightarrow\mathbb{C}[S/\alpha]$. Since every ideal of $\mathbb{C}[S]/J(\mathbb{C}[S])\cong \bigoplus_{i=1}^kM_{n_i}(\mathbb{C})$ is a direct summand, then we have
$$\mathbb{C}[S]/J(\mathbb{C}[S])\cong \ker(\phi)\oplus\mathbb{C}[S/\alpha].$$
 By Lemma \ref{direkt} we get
 $$\mathbb{C}[S]/J(\mathbb{C}[S])\cong \ker(\phi)\oplus\underbrace{\mathbb{C}\oplus\dots\oplus\mathbb{C}}_{|S/\alpha| \text{ times}},$$
 from which we can conclude that if $S$ is not s-indecomposable then $\mathbb{C}[S/K_S]/J(\mathbb{C}[S/K_S])$ has more than one $1$-dimensional ideal.

In the next we show that if $S$ is s-indecomposable then the semigroup algebra $\mathbb{C}[S]/J(\mathbb{C}[S])$ has exactly one $1$-dimensional ideal. Let $S$ be a finite s-indecomposable semigroup (with a zero $z$). The factor algebra $\mathbb{C}[S]/J(\mathbb{C}[S])$ is semisimple in which $\mathbb{C}[z]+J(\mathbb{C}[S])$ is a $1$-dimensional ideal. 
To show that this is the only $1$-dimensional ideal of $\mathbb{C}[S]/J(\mathbb{C}[S])$, it is sufficient to show that $\mathbb{C}[S]/(J(\mathbb{C}[S])+\mathbb{C}[z])$ does not contain $1$-dimensional ideals.

Denote $A:=\mathbb{C}[S]/(J(\mathbb{C}[S])+\mathbb{C}[z])$. We will show $A\cong \mathbb{C}_0[S]/J(\mathbb{C}_0[S])$.
It is easy to see that $J(\mathbb{C}[S])\cap \mathbb{C}[z]=0$. We know that $\mathbb{C}[S]\cong \mathbb{C}_0[S]\oplus\mathbb{C}[z]$ (\cite[Cor. 9 p.38]{Okninski}).
So
$$A\cong (\mathbb{C}_0[S]\oplus \mathbb{C}[z])/(J(\mathbb{C}[S])\oplus\mathbb{C}[z])\cong \mathbb{C}_0[S]/J(\mathbb{C}_0[S]).$$

Since $A$ is semisimple, we get 
$$A\cong \bigoplus_{i=1}^kM_{n_i}(\mathbb{C}).$$
Suppose indirectly that $n_j=1$ for some $j$ with $1\leq j\leq k$. Denote the composition of canonical homomorphisms $\mathbb{C}[S]\rightarrow \mathbb{C}_0[S]$ and $\mathbb{C}_0[S]\rightarrow A$ by $\phi$. Let $\pi:A\rightarrow M_{n_j}(\mathbb{C})\cong \mathbb{C}$ be the canonical projection. 
Let 
$$I:=\ker(\pi\circ\phi)\cap S=\{s\in S|\pi(\phi(s))=0\}.$$

It is easy to see that $\phi(S)$ generates $A$ but $\phi(I)$ does not. Hence $S\ne I$. $\phi(z)=0$ so $z\in I$ which means $I$ is a nonempty proper subset of $S$. It is easy to check that $I$ is a completely prime ideal of $S$ which contradicts the assumption that $S$ is s-indecomposable. Thus our assertion is proved in that case when S has a zero element.

As a finite semigroup $S$ is s-indecomposable if and only if $S/K_S$ is s-indecomposable, then the general case is an easy corollary of the case when $S$ has a zero.
\end{proof}

\begin{remark}
In case of finite semigroups with a zero, the s-indecomposability can be described only with the property of the semigroup algebra (Theorem \ref{tetel}). It is not true for finite semigroups in general. For example if $G$ is a finite Abelian group and $Y$ is a finite semilattice such that $|G|=|Y|$ then $\mathbb{C}[G]\cong \bigoplus_{i\in G}\mathbb{C}\cong \mathbb{C}[Y]$. Thus, if $1<|G|=|Y|$, then $G$ is s-indecomposable, $Y$ is not, but $\mathbb{C}[G]\cong \mathbb{C}[Y]$.
\end{remark}

\section{Embeddings into semilattice indecomposable semigroups}

Let $A, B$ semigroups with zeros $z_A,z_B$. $A\times B$ has an ideal $I=(\{z_A\}\times B) \cup (A\times \{z_B\})$. Let $A \times_0 B$ denote the Rees factor semigroup $(A\times B)/I$.

\begin{proposition}\label{nulltimes}
For arbitrary semigroups $A$ and $B$ with zeros, the semigroup $A\times _0 B$ is s-indecomposable if and only if $A$ or $B$ is s-indecomposable.
\end{proposition}

\begin{proof} Let $A$ and $B$ be arbitrary semigroups with zeros $z_A$ and $z_B$. Assume that $A$ is s-indecomposable. Consider the semigroup $A\times _0 B$. We show that $A\times _0 B$ is s-indecomposable. Let $\varphi$ denote the canonical homomorphism of $A\times B$ onto $A\times _0B$. Let $x, y\in A\times _0B$ be arbitrary elements. Let $(a_x, b_x)$ and $(a_y, b_y)$ be elements of $A\times B$ such that $\varphi ((a_x, b_x))=x$ and $\varphi ((a_y, b_y))=y$.
As $A$ is s-indecomposable, there are elements $a_x=a_1, \dots , a_t=z_A$ and $z_A=a_t, a_{t+1}, \cdots, a_k=a_y$ such that $a_i$ divides some power of $a_{i+1}$ for every $i=1, \dots , k-1$ (\cite[Cor.]{Tamura2}). From this it follows that
\[(a_1; b_x), \dots , (a_t;b_x)=(z_A;b_x)\] and
\[(z_A;b_y)=(a_t;b_y), \dots , (a_k;b_y)\] are sequences of $A\times B$ such that every elements of the sequences (except the last) divides some power of the next. Then
\[x=\varphi((a_1; b_x)), \dots , \varphi((a_t;b_x))=\varphi((z_A;b_x))\] and
\[\varphi((z_A;b_y))=\varphi((a_t;b_y)), \dots , \varphi((a_k;b_y))=y\]
are sequences of $A\times _0B$ such that every elements of the sequences (except the last) divides some power of the next. As $\varphi((z_A;b_x))=\varphi((z_A;b_y))$,
we get that
\[x= \varphi((a_1, b_x)),\dots \varphi ((z_A,b_x))=\varphi((z_A;b_y)), \dots ,\varphi((a_k,b_y))=y\]
is a sequence of $A\times _0B$ such that every elements of the sequence (except the last) divides some power of the next. Then, by Corollary of \cite{Tamura2}, $A\times _0B$ is s-indecomposable.
The proof is similar in that case when the semigroup $B$ is s-indecomposable.

Conversely, assume that $A\times _0B$ is s-indecomposable. If $A$ and $B$ are  not s-indecomposable then there are completely prime ideals $P_A\subset A$ and $P_B\subset B$. It is easy to see that $\varphi((P_A\times B)\cup (A\times P_B))$ is a proper completely prime ideal of $A\times _0B$ and so $A\times _0B$ is not s-indecomposable. This is a contradiction, hence $A$ or $B$ must be s-indecomposable.
\end{proof}

\begin{remark}\rm We have a different proof of Proposition \ref{nulltimes} in finite case. Suppose $A$ and $B$ are finite.
By \cite[Cor. 9 p.39 and Lemma 10 p.40]{Okninski} we get:

$$\mathbb{C}[A\times_0 B]\cong \mathbb{C}\oplus\mathbb{C}_0[A\times_0 B]\cong \mathbb{C}\oplus (\mathbb{C}_0[A]\otimes\mathbb{C}_0[B]),$$
so
$$\mathbb{C}[A\times_0 B]/J(\mathbb{C}[A\times_0 B])\cong \mathbb{C}\oplus \left((\mathbb{C}_0[A]/J(\mathbb{C}_0[A]))\otimes\mathbb{C}_0[B]/J(\mathbb{C}_0[B])\right).$$
So $\mathbb{C}[A\times_0 B]/J(\mathbb{C}[A\times_0 B])$ has exactly one $1$-dimensional ideal if and only if $\mathbb{C}_0[A]/J(\mathbb{C}_0[A])$ or $\mathbb{C}_0[B]/J(\mathbb{C}_0[B])$ has no $1$-dimensional ideal. $\mathbb{C}_0[A]/J(\mathbb{C}_0[A])$ has no $1$-dimensional ideal if and only if $\mathbb{C}[A]/J(\mathbb{C}[A])$ has exactly one. By Theorem \ref{tetel} we get the statement.
\end{remark}

We will see that the smallest s-indecomposable semigroup which contains a $2$-element subsemilattice has $5$ elements. Moreover the smallest s-indecomposable semigroup which contains a $3$-element subsemilattice is isomorphic to the semigroup $B_2$ (Theorem \ref{theorem2}, Theorem \ref{theorem3}). First we show that there are only two nonisomorphic $5$-element s-indecomposable semigroup with a $2$-element subsemilattice.

\begin{corollary}\label{B2}
Let $S$ be an s-indecomposable semigroup such that $|S|\leq 5$, and $S$ has at least two commuting idempotents. Then $S\cong B_2$ or $S\cong \mathcal{M}^0(1;2,2;P)$, where $P=\begin{bmatrix}
1 & 1\\
0 & 1
\end{bmatrix}$.
\end{corollary}

\begin{proof} By Theorem \ref{tetel} we get $\mathbb{C}[S/K_S]/J(\mathbb{C}[S/K_S])$ is isomorphic to $\mathbb{C}$ or $\mathbb{C}\oplus M_2(\mathbb{C})$. If $\mathbb{C}[S/K_S]/J(\mathbb{C}[S/K_S])\cong \mathbb{C}$ then $\mathbb{C}_0[S/K_S]$ is nilpotent. If $\mathbb{C}_0[S/K_S]$ is nilpotent then all idempotents of $S$ contained in $K_S$. So there are two commuting idempotents in $K_S$ which contradicts the fact that $K_S$ is completely simple. Hence $\mathbb{C}[S/K_S]/J(\mathbb{C}[S/K_S])\cong \mathbb{C}\oplus M_2(\mathbb{C})$. Moreover $\dim(\mathbb{C}[S/K_S]/J(\mathbb{C}[S/K_S]))=\dim(\mathbb{C}[S])$ and so $J(\mathbb{C}[S/K_S])=0$ and $|K_S|=1$. Thus $S$ has a zero $z$ and $\mathbb{C}[S]\cong \mathbb{C}\oplus M_2(\mathbb{C})$. If $I$ is an ideal of $S$ then $\mathbb{C}[I]$ is an ideal of $\mathbb{C}[S]$.
$\mathbb{C}[S]$ has exactly two proper ideals: one of them is the augmentation ideal (see \cite[p.35]{Okninski}) and the other is spanned by $z$. Consequently $S$ is a (finite) $0$-simple semigroup, so it is completely $0$-simple. Using the Rees Theorem, it is a matter of checking to see that $S$ is isomorphic to one of the two semigroups listed in the corollary.
\end{proof}

\begin{corollary}
Every finite semigroup $S$ can be embedded into an s-indecomposable semigroup containing $4|S|+1$ elements.
\end{corollary}

\begin{proof}
Let $S$ be an arbitrary finite semigroup. Denote $S^0$ the semigroup $S$ with a zero adjoined (also in that case when $S$ has a zero). Let $T$ be a $5$-element s-indecomposable semigroup with a zero and an other idempotent $e$ (these semigroups are described in Corollary \ref{B2}). Since $T$ is an s-indecomposable semigroup with a zero, then $S^0\times_0 T$ is s-indecomposable by Proposition \ref{nulltimes}. Moreover $|S^0\times_0 T|=4|S|+1$. Let $\varphi$ denote the canonical homomorphism of $S^0\times T$ onto $S^0\times _0 T$. Define the homomorphism $\pi:S\rightarrow S^0\times T$ by $\pi(s):=(s,e)$.
It is obvious that $\varphi\circ\pi$ is an embedding of $S$ into the s-indecomposable semigroup $S^0\times_0 T$.
\end{proof}

\begin{proposition}\label{plusone}
Every finite s-indecomposable semigroup $S$ with a zero can be embedded into an s-indecomposable semigroup containing $|S|+1$ elements.
\end{proposition}

\begin{proof}
Let $z\in S$ be the zero. Define $S'$ the semigroup which can be obtained from $S$ by adjunction of an element $z'$, such that $z'x:=z$, $xz':=z$, and $(z')^2:=z$ where $x$ is an arbitrary element of $S$. Then $z-z'\in J(\mathbb{C}[S])$ so $\mathbb{C}[S]/J(\mathbb{C}[S])\cong \mathbb{C}[S']/J(\mathbb{C}[S'])$. Since $S$ is s-indecomposable by Theorem \ref{tetel}, we get that $S'$ is also s-indecomposable.
\end{proof}

\section{On the cardinality of subsemilattices of semilattice indecomposable finite semigroups}

In this section we answer the question: what is the
cardinality of subsemilattices of s-indecomposable finite semigroups. First we deal with the case when the considered semigroup has a zero (Proposition \ref{bound}). Then we consider the arbitrary case (Theorem \ref{theorem2}).

\begin{proposition}\label{bound}
Let $S$ be an s-indecomposable finite semigroup with a zero. If $Y$ is a subsemilattice of $S$ then $|Y|\leq 2\left\lfloor\frac{|S|-1}{4}\right\rfloor+1$.
\end{proposition}

\begin{proof}
By Theorem \ref{tetel} $ \mathbb{C}[S]/J(\mathbb{C}[S])\cong \mathbb{C}\oplus\bigoplus_{i=1}^kM_{n_i}(\mathbb{C})$ such that $n_i\ne 1$ ($i=1\dots k$). By Lemma \ref{direkt}, $J(\mathbb{C}[S])\cap \mathbb{C}[Y]=0$ and so $\mathbb{C}[S]/J(\mathbb{C}[S])$ has a subalgebra which is isomorphic to $\mathbb{C}[Y]$. So it contains $|Y|$ commuting linearly independent projections, which are simultaneously diagonalizable. Thus
\begin{equation}\label{equation1}
|Y|\leq 1+\sum_{i=1}^kn_i\leq 1+\sum_{i=1}^{\left\lfloor\frac{|S|-1}{4}\right\rfloor}2=2\left\lfloor\frac{|S|-1}{4}\right\rfloor+1.
\end{equation} 

\end{proof}
 
\begin{theorem}\label{theorem2}
Let $S$ be an s-indecomposable finite semigroup. 
\begin{enumerate}[(i)]
\item If $Y$ is a subsemilattice of $S$ then $|Y|\leq 2\left\lfloor\frac{|S|-1}{4}\right\rfloor+1$.

\item For every positive integer $n$, there is a semigroup $S$ such $|S|=n$ and there is a subsemilattice $Y$ of $S$ such that $|Y|= 2\left\lfloor\frac{|S|-1}{4}\right\rfloor+1$.
\end{enumerate}
\end{theorem}

\begin{proof}
\begin{enumerate}[(i)]

\item $K_S$ is a finite completely simple semigroup. So if $e,f\in K_S$ are commuting idempotents then $e=f$. Thus 
$$|Y\cap K_S|\leq 1,$$
and so $Y\cong Y/(Y\cap K_S)$.
Obviously $S/K_S$ is an s-indecomposable semigroup with a zero, and $Y/(Y\cap K_S)$ is a subsemilattice of $S/K_S$. By Proposition \ref{bound}, we get:
$$|Y/(Y\cap K_S)|\leq 2\left\lfloor\frac{|S/K_S|-1}{4}\right\rfloor+1.$$
Thus
\begin{equation}\label{equation2}
\begin{split}
|Y|=|Y/(Y\cap K_S)|\leq 2\left\lfloor\frac{|S/K_S|-1}{4}\right\rfloor+1\leq 2\left\lfloor\frac{|S|-1}{4}\right\rfloor+1.
\end{split}
\end{equation}

\item Let $n$ be a positive integer. We can consider $n$ in the form $n=4k+1+l$ where $0\leq l<4$. Let $Y$ be a semilattice such that $|Y|=k+1$. By Proposition \ref{nulltimes}, $Y\times_0 B_2$ is s-indecomposable, because $B_2$ is s-indecomposable (Corollary \ref{B2}). $B_2$ has a $3$-element subsemilattice $V$. So $Y\times_0 B_2$ has a subsemilattice $Y\times_0 V$, it has $2k+1$ elements. Applying $l$ times the embedding of Proposition \ref{plusone} to $Y\times_0 B_2$, we get an $n$-element s-indecomposable semigroup in which $Y\times_0V$ is a subsemilattice containing $2k+1=2\left\lfloor\frac{|S|-1}{4}\right\rfloor+1$ elements.
\end{enumerate}
\end{proof}

In this paper we deal with only that s-indecomposable semigroups $S$ containing a subsemilattice $Y$ with $|Y|=2\left\lfloor\frac{|S|-1}{4}\right\rfloor+1$ which contain $4k+1$ elements. In the next section we describe the structure of these ones.

\section{$B_2$-combinatorial semigroups}

\begin{definition}\label{def}
A semigroup $S$ is said to be $B_2$-combinatorial if $S$ is s-indecomposable, $|S|=4k+1$ ($k$ is a nonnegative integer) and $S$ has a subsemilattice $Y$ with  $|Y|=2\left\lfloor\frac{|S|-1}{4}\right\rfloor+1=\frac{|S|+1}{2}=2k+1$.
\end{definition}

The name $B_2$-combinatorial will be clear at Theorem \ref{theorem3}. First of all we note that the semigroup $B_2$ is $B_2$-combinatorial.

\begin{proposition}\label{ex}
Let $S$ be a $B_2$-combinatorial semigroup. Then all of the following assertions hold. 
\begin{enumerate}[(i)]
\item $S$ has a zero.
\item The semigroup algebra $\mathbb{C}[S]$ is isomorphic to $ \mathbb{C}\oplus\bigoplus_{i=1}^k M_2(\mathbb{C})$.
\item Every ideal of $S$ is $B_2$-combinatorial.
\item Every homomorphic image of $S$ is $B_2$-combinatorial.
\end{enumerate}
\end{proposition}

\begin{proof} Let $S$ be a $B_2$-combinatorial semigroup and let $Y$ denote a subsemilattice of $S$ with $|Y|=2\left\lfloor\frac{|S|-1}{4}\right\rfloor+1$. 
\begin{enumerate}[(i)]
\item By (\ref{equation2}) in proof of Theorem \ref{theorem2} we have $\left\lfloor\frac{|S/K_S|-1}{4}\right\rfloor=\left\lfloor\frac{|S|-1}{4}\right\rfloor$ and so $|S/K_S|=|S|$ thus $|K_S|=1$. Hence $S$ has a zero.
\item If in the proof of Proposition \ref{bound} inequation (\ref{equation1}) is an equation then 
$$\mathbb{C}[S]/J(\mathbb{C}[S])\cong \mathbb{C}\oplus\bigoplus_{i=1}^{\left\lfloor\frac{|S|-1}{4}\right\rfloor} M_2(\mathbb{C}).$$
Hence $\dim(\mathbb{C}[S]/J(\mathbb{C}[S]))=\dim(\mathbb{C}[S])$ which means $J(\mathbb{C}[S])=0$. For $k=\left\lfloor\frac{|S|-1}{4}\right\rfloor$ we get
$$\mathbb{C}[S]\cong \mathbb{C}\oplus\bigoplus_{i=1}^k M_2(\mathbb{C}).$$
\item 
Let $I$ be an ideal of $S$. Since every ideal of an s-indecomposable semigroup is also an s-indecomposable (\cite[Lemma 4]{Tamura1}), then $I$ is s-indecomposable. 
As $Y$ is a possible greatest subsemilattice of $S$, the zero of $S$ is in $Y$. Hence $Y\cap I\ne \emptyset$. It is clear that $Y\cap I$ and $Y/(Y\cap I)$ are subsemilattices of $I$ and $S/I$ respectively. Since $I$ and $S/I$ are s-indecomposable semigroups with zeros, then we can use (i) of Theorem \ref{theorem2}. 
Hence $$|Y\cap I|\leq 2\left\lfloor\frac{|I|-1}{4}\right\rfloor+1\text{ and }|Y/(Y\cap I)|\leq 2\left\lfloor\frac{|S/I|-1}{4}\right\rfloor+1.$$
Moreover
$|Y|=|Y\cap I|+|Y/(Y\cap I)|-1$, $|Y|=\frac{|S|+1}{2}$ and $|S|=|I|+|S/I|-1=4k+1.$

From the previous equations and inequations we get

$$0\leq \left(\left\lfloor\frac{|I|-1}{4}\right\rfloor+\left\lfloor\frac{1-|I|}{4}\right\rfloor\right).$$

From this it follows that $|I|=4l+1$.
Then $|Y\cap I|=2\left\lfloor\frac{|I|-1}{4}\right\rfloor+1$. Indeed, if we suppose indirectly $|Y\cap I|< 2\left\lfloor\frac{|I|-1}{4}\right\rfloor+1$, then we can get that
$$0< \left(\left\lfloor\frac{|I|-1}{4}\right\rfloor+\left\lfloor\frac{1-|I|}{4}\right\rfloor\right),$$
which is a contradiction. Hence $|Y\cap I|= 2\left\lfloor\frac{|I|-1}{4}\right\rfloor+1=2l+1$. $I$ is an s-indecomposable semigroup with $|I|=4l+1$, where $l$ is an integer with $0\leq l\leq k$, and  $Y\cap I$ is a subsemilattice of $I$ with $|Y\cap I|=2l+1$. Thus $I$ is a $B_2$-combinatorial semigroup.

\item
Let $\phi:S\rightarrow T$ be a surjective homomorphism. 

Since every homomorphic image of an s-indecomposable semigroup is also an s-indecomposable semigroup (\cite[Lemma 3]{Tamura1}), then $T$ is s-indecomposable.

Extend $\phi$ to an algebra homomorphism $\hat{\phi}:\mathbb{C}[S]\rightarrow \mathbb{C}[T]$. By (i) $S$ is a semigroup with a zero $z$ so $\phi(z)$ is a zero of $T$. By Corollary 9 of \cite[p.38]{Okninski}, we get
$$\mathbb{C}[T]\cong \mathbb{C}\oplus\mathbb{C}_0[T].$$
 Since $\mathbb{C}_0[S]\cong \bigoplus_{i=1}^k M_2(\mathbb{C})$  for $k=\frac{|S|-1}{4}$ (see (ii)) and $\mathbb{C}_0[T]$ is homomorphic image of $\mathbb{C}_0[S]$ then we get 
 $$\mathbb{C}_0[T]\cong\bigoplus_{i=1}^l M_2(\mathbb{C})$$
for $l=\frac{|T|-1}{4}$, and $|T|=\dim(\mathbb{C}[T])=4l+1$.  So $\ker(\hat{\phi})\cong \bigoplus_{i=1}^{k-l} M_2(\mathbb{C})$.
 
In a suitable basis of $\mathbb{C}[S]\cong \mathbb{C}\oplus \bigoplus_{i=1}^k M_2(\mathbb{C})$ the subalgebra $\mathbb{C}[Y]$ consists of all the diagonal matrices of $\mathbb{C}\oplus \bigoplus_{i=1}^k M_2(\mathbb{C})$. It is easy to see that
$$\dim(\mathbb{C}[Y]\cap \ker(\hat{\phi}))=2(k-l).$$
Thus
$$|Y/\ker(\phi|_Y)|=\dim(\mathbb{C}[Y]/(\mathbb{C}[Y]\cap \ker(\hat{\phi})))=2l+1.$$
Hence $T$ is an s-indecomposable semigroup with $|T|=4l+1$ and $Y/\ker(\phi|_Y)$ is a subsemilattice of $T$ with $|Y/\ker(\phi|_Y)|=2l+1$. It means that $T$ is a $B_2$-combinatorial semigroup.
\end{enumerate}
\end{proof}

\begin{lemma}\label{comp}
Let $S$ be a completely $0$-simple semigroup and $Y$ a subsemilattice of $S$. Then $|Y|\leq \sqrt{|S|-1}+1$. If $|Y|= \sqrt{|S|-1}+1$ then $S\cong \mathcal{M}^0(1;n,n;I)$, where $n=\sqrt{|S|-1}$. 
\end{lemma}

\begin{proof}
Let $S$ be a completely $0$-simple semigroup and $Y$ be a subsemilattice of $S$. Let $k:=|Y|-1$.
By the Rees Theorem, $S$ is isomorphic to a Rees matrix semigroup $\mathcal{M}^0(G;n,m;P)$. The nonzero idempotents of $S$ are in the form $((P_{j,i})^{-1};i,j)$ with $P_{j,i}\ne 0$. If $(g;i,j), (h;k,l)$ are different commutable nonzero idempotents of $S$ then $P_{j,k}=0$ and $P_{l,i}=0$. It means that there is a $k\times k$ permutation matrix $R$ and a $k\times k$ diagonal matrix $D$ over $G^0$ such that $RD$ is submatrix of $P$. From $k\leq \min\{n,m\}$ and $|S|=|G|nm+1$ we get $$k\leq \sqrt{|S|-1}.$$
This inequation is an equation if and only if $|G|=1$, $n=m$ and every row and every column of $P$ has exactly one nonzero element. Using the Lemma 3.6. of \cite[p.94]{CliffordPreston}, we get that the inequation is equation if and only if $S\cong \mathcal{M}^0(1;n,n;I)$.
\end{proof}

\begin{proposition}\label{B_2}
If $S$ is a $B_2$-combinatorial $0$-simple semigroup then $S\cong B_2$. 
\end{proposition}

\begin{proof}
Since $S$ is $B_2$-combinatorial, then it has a subsemilattice $Y$ such $|Y|=\frac{|S|+1}{2}$. By Lemma \ref{comp}, we get $|Y|\leq \sqrt{|S|-1}+1$. From $$\frac{|S|+1}{2}\leq \sqrt{|S|-1}+1$$ we get $1\leq |S|\leq 5$. Since $S$ is $B_2$-combinatorial, then $|S|=4k+1$ for a nonnegative integer $k$. The one-element semigroup is not $0$-simple and so $|S|=5$. Since $S$ is $B_2$-combinatorial, then $S$ has a subsemilattice $Y$ with $|Y|=3$. By Lemma \ref{comp}, we get $S\cong \mathcal{M}^0(1;2,2;I)$.
\end{proof}

Let $S$ be a semigroup. If $J(a)$ denotes the principal ideal of $S$ generated by an element $a\in S$, then $I(a)=\{b|b\in S;J(a)\ne J(b)\}$ is either empty or ideal of $S$. The factor semigroup $J(a)/I(a)$ is called a principal factor of $S$. It is known that every principal factor of any semigroup is $0$-simple, simple or null (\cite[Lemma 2.39. p.73]{CliffordPreston}).

\begin{theorem}\label{theorem3} Let $S$ be a finite semigroup. Then (i) and (ii) are equivalent:
\begin{enumerate}[(i)]
\item $S$ is a $B_2$-combinatorial semigroup,
\item $S$ has a zero and, for every nonzero element $a$ of $S$, the principal factor $J(a)/I(a)$ is isomorphic to the semigroup $B_2$.
\end{enumerate}
\end{theorem}

\begin{proof}
(i)$\Rightarrow$(ii) Let $S$ be a $B_2$-combinatorial semigroup. By (i) of Proposition \ref{ex}, $S$ has a zero. Let $a$ be a nonzero element of $S$. By (iii) and (iv) of Proposition \ref{ex}, $J(a)/I(a)$ is $B_2$-combinatorial. It is easy to see that a simple semigroup or a null semigroup is $B_2$-combinatorial if and only if it contains exactly one element. Hence $J(a)/I(a)$ is $0$-simple. By Proposition \ref{B_2}, $J(a)/I(a)\cong B_2$.

(ii)$\Rightarrow$(i): Every principal factor of $S$ is an inverse semigroup, hence $S$ is also an inverse semigroup. Its idempotents form a semilattice and the number of idempotents is $\frac{|S|+1}{2}$. Since every principal factor is s-indecomposable and has a divisor of zero, we get that $S$ is s-indecomposable. 
\end{proof}

The previous theorem shows that why we use the expression "$B_2$-combinatorial" for semigroups defined in Definition \ref{def}. These semigroups are combinatorial inverse semigroups and the principal factors defined by nonzero elements are isomorphic to the semigroup $B_2$.

\medskip
The $B_2$-combinatorial semigroups are combinatorial ($\mathcal{H}$ relation is identical) inverse semigroups.
If $Y$ is a semilattice then $Y\times_0 B_2$ is a $B_2$-combinatorial semigroup. The next example shows that not all $B_2$-combinatorial semigroups can be constructed in this way.
\begin{example}\rm

If $E$ is a semilattice then $T_E$ denote the Munn semigroup of $E$ described in \cite[p.162]{Howie}. Consider the following semilattices:

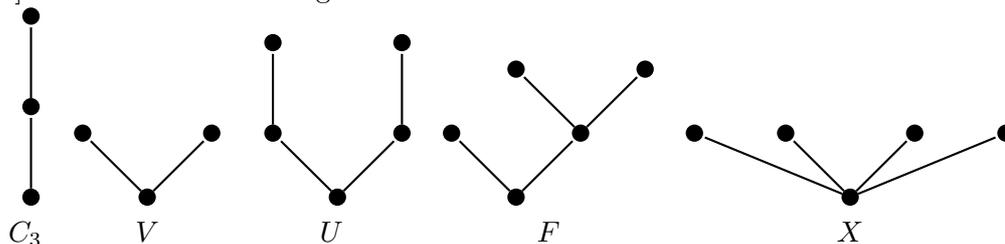
\begin{tabular}{c c c c c}
\begin{tikzpicture}[->,>=stealth',shorten >=1pt,auto,node distance=1.2cm,
        thick,main node/.style={circle,fill=black,draw,minimum size=0.2cm,inner sep=0pt]}]

    \node[main node] (1) {};
    \node[main node] (2) [above of=1]  {};
    \node[main node] (3) [above of=2] {};

    \path[-]
    (1) edge node {} (2)
    (2) edge node {} (3);
\end{tikzpicture}&

\begin{tikzpicture}[->,>=stealth',shorten >=1pt,auto,node distance=1.2cm,
        thick,main node/.style={circle,fill=black,draw,minimum size=0.2cm,inner sep=0pt]}]

    \node[main node] (1) {};
    \node[main node] (2) [above left of=1]  {};
    \node[main node] (3) [above right of=1] {};

    \path[-]
    (1) edge node {} (2)
    (1) edge node {} (3);
\end{tikzpicture}&

\begin{tikzpicture}[->,>=stealth',shorten >=1pt,auto,node distance=1.2cm,
        thick,main node/.style={circle,fill=black,draw,minimum size=0.2cm,inner sep=0pt]}]

    \node[main node] (1) {};
    \node[main node] (2) [above left of=1]  {};
    \node[main node] (3) [above right of=1] {};
    \node[main node] (4) [above of=2]  {};
    \node[main node] (5) [above of=3] {};

    \path[-]
    (1) edge node {} (2)
    (1) edge node {} (3)
    (2) edge node {} (4)
    (3) edge node {} (5);
    
\end{tikzpicture}&

\begin{tikzpicture}[->,>=stealth',shorten >=1pt,auto,node distance=1.2cm,
        thick,main node/.style={circle,fill=black,draw,minimum size=0.2cm,inner sep=0pt]}]

    \node[main node] (1) {};
    \node[main node] (2) [above left of=1]  {};
    \node[main node] (3) [above right of=1] {};
    \node[main node] (4) [above left of=3]  {};
    \node[main node] (5) [above right of=3] {};

    \path[-]
    (1) edge node {} (2)
    (1) edge node {} (3)
    (3) edge node {} (4)
    (3) edge node {} (5);
    
\end{tikzpicture}&

\begin{tikzpicture}[->,>=stealth',shorten >=1pt,auto,node distance=1.2cm,
        thick,main node/.style={circle,fill=black,draw,minimum size=0.2cm,inner sep=0pt]}]

    \node[main node] (1) {};
    \node[main node] (2) [above left of=1]  {};
    \node[main node] (3) [above right of=1] {};
    \node[main node] (4) [left of=2]  {};
    \node[main node] (5) [right of=3] {};

    \path[-]
    (1) edge node {} (2)
    (1) edge node {} (3)
    (1) edge node {} (4)
    (1) edge node {} (5);
    
\end{tikzpicture}\\
 $C_3$ & $V$ & $U$ & $F$ & $X$
\end{tabular}

We show that there are only $3$ nonisomorphic $B_2$-combinatorial semigroups containing $9$ elements:

$C_3\times_0 B_2\cong T_U$, $V\times_0 B_2\subset T_X$ and $T_F$.
\begin{proof}
If a semigroup is a combinatorial inverse semigroup then it is a fundamental inverse semigroup (inverse semigroup such the maximum idempotent separating congruence is identical)\cite[Prop. 5.3.7 p.161]{Howie}. 
Let $S$ be a $B_2$-combinatorial semigroup containing $9$ elements. Then $|E(S)|=5$, where $E(S)$ denotes the set of all idempotents of $S$. So $S$ is isomorphic to a full inverse subsemigroup of the Munn semigroup of $E(S)$ (\cite[Thm.5.4.5 p.165]{Howie}).  
It is a  matter of checking to see that there are only three Munn semigroup containing a full inverse subsemigroup with $9$ elements which is $B_2$-combinatorial: $T_U, T_F$ and $T_X$. The semigroups $T_U$ and $T_F$ are $B_2$-combinatorial $9$-element semigroups. $T_X$ has three $9$-element $B_2$-combinatorial subsemigroups, these are isomorphic to $V\times_0 B_2$.
\end{proof}

Suppose that there is a semilattice $Y$ such that $T_F\cong Y\times_0 B_2$. Since $|T_F|=9$, then we get $|Y|=3$. The $3$-element nonisomorphic semilattices are $C_3$ and $V$. It is a matter of checking to see that
$$E(C_3\times_0 B_2)\cong U, E(V\times_0 B_2)\cong X\text{ and }E(T_F)\cong F.$$
Consequently there is no semilattice $Y$ such that $Y\times_0 B_2\cong T_F$.

\end{example}

\begin{center} Acknowledgement \end{center} 

The author is grateful to A. Nagy for some discussions and advices.

\end{document}